\documentclass[]{rmj-public-adapted}
\pdfoutput=1

\usepackage{latexsym}     
\usepackage{amssymb}      
\usepackage{booktabs}     
\usepackage[all]{xy}      
\usepackage{colonequals}  
\usepackage{pdflscape}    
\usepackage{afterpage}    

\newtheorem{theorem}{Theorem}[section]
\newtheorem{lemma}[theorem]{Lemma}
\newtheorem{definition}[theorem]{Definition}
\newtheorem{notation}[theorem]{Notation}
\newtheorem{convention}[theorem]{Convention}

\newcommand{\ZZ}{{\mathbf Z}}

\newcommand{\pz}{\phantom{0}}        
\newcommand{\ellipsis}{
.\kern0.08em
.\kern0.08em
.\@\kern0.08em\relax}

\makeatletter
\@namedef{subjclassname@2020}{%
  \textup{2020} Mathematics Subject Classification}
\makeatother

\title[Powers of $3$ with few nonzero bits]
      {Powers of $3$ with few nonzero bits \\ and a conjecture of Erd\H{o}s}
      
\author[Dimitrov]{Vassil S. Dimitrov}
\address[Dimitrov]{
Center for Information Security and Cryptography,
University of Calgary,
2500 University Drive NW,
Calgary, AB T2N 1N4, Canada}
\email{\href{mailto:vdimitro@ucalgary.ca}{vdimitro@ucalgary.ca}}

\address[Dimitrov]{
Lemurian Labs, Inc. 
}
\email{\href{mailto:vassil@lemurianlabs.com}{ vassil@lemurianlabs.com}}

\author[Howe]{Everett W. Howe}
\address[Howe]{Independent mathematician, 
         San Diego, CA 92104, USA}
\email{\href{mailto:however@alumni.caltech.edu}{however@alumni.caltech.edu}}
\urladdr{\href{http://ewhowe.com}{http://ewhowe.com}}

\date{May 12, 2023}



\begin{document}
\begin{abstract}
Using completely elementary methods, we find all powers of $3$ that can be 
written as the sum of at most twenty-two distinct powers of~$2$, as well as all 
powers of $2$ that can be written as the sum of at most twenty-five distinct
powers of~$3$. The latter result is connected to a conjecture of Erd\H{o}s,
namely, that $1$, $4$, and $256$ are the only powers of $2$ that can be written
as a sum of distinct powers of~$3$.

We present this work partly as a reminder that for certain exponential 
Diophantine equations, elementary techniques based on congruences can yield 
results that would be difficult or impossible to obtain with more advanced
techniques involving, for example, linear forms in logarithms.
\end{abstract}

\maketitle


\section{Introduction}
\label{S:intro}

To introduce our topic, we begin with some numerical observations. For an 
integer $x\ge 0$, consider the binary representation of $3^x$. In 
Table~\ref{T:binary} we give this representation for $x \le 25$, and we tabulate
the number of bits in the binary representation together with the number of 
those bits that are equal to~$1$.

\begin{table}
\label{T:binary}
\caption{For each $x$ between $0$ and $25$ we give the binary representation 
         of~$3^x$, together with the total number of bits in the representation
         and the number of those bits that are equal to~$1$.}
\begin{tabular}{r r c c c}
\toprule
$x$  &             Binary representation of $3^x$ &&  \#Bits  & \#Ones  \\
\midrule
 $0$ &                                        $1$ && $\pz 1$  & $\pz 1$ \\
 $1$ &                                       $11$ && $\pz 2$  & $\pz 2$ \\
 $2$ &                                     $1001$ && $\pz 4$  & $\pz 2$ \\
 $3$ &                                    $11011$ && $\pz 5$  & $\pz 4$ \\
 $4$ &                                  $1010001$ && $\pz 7$  & $\pz 3$ \\
 $5$ &                                 $11110011$ && $\pz 8$  & $\pz 6$ \\
 $6$ &                               $1011011001$ &&    $10$  & $\pz 6$ \\
 $7$ &                             $100010001011$ &&    $12$  & $\pz 5$ \\
 $8$ &                            $1100110100001$ &&    $13$  & $\pz 6$ \\
 $9$ &                          $100110011100011$ &&    $15$  & $\pz 8$ \\
$10$ &                         $1110011010101001$ &&    $16$  & $\pz 9$ \\
$11$ &                       $101011001111111011$ &&    $18$  &    $13$ \\
$12$ &                     $10000001101111110001$ &&    $20$  &    $10$ \\
$13$ &                    $110000101001111010011$ &&    $21$  &    $11$ \\
$14$ &                  $10010001111101101111001$ &&    $23$  &    $14$ \\
$15$ &                 $110110101111001001101011$ &&    $24$  &    $15$ \\
$16$ &               $10100100001101011101000001$ &&    $26$  &    $11$ \\
$17$ &              $111101100101000010111000011$ &&    $27$  &    $14$ \\
$18$ &            $10111000101111001000101001001$ &&    $29$  &    $14$ \\
$19$ &          $1000101010001101011001111011011$ &&    $31$  &    $17$ \\
$20$ &         $11001111110101000001101110010001$ &&    $32$  &    $17$ \\
$21$ &       $1001101111011111000101001010110011$ &&    $34$  &    $20$ \\
$22$ &      $11101001110011101001111100000011001$ &&    $35$  &    $19$ \\
$23$ &    $1010111101011010111101110100001001011$ &&    $37$  &    $22$ \\
$24$ &  $100000111000010000111001011100011100001$ &&    $39$  &    $16$ \\
$25$ & $1100010101000110010101100010101010100011$ &&    $40$  &    $18$ \\
\bottomrule
\end{tabular}
\end{table}

Based on this limited data, it looks like about half of the bits of the binary
representation of $3^x$ are equal to $1$, which is what one would expect if 
$3^x$ were to behave like a random integer of the appropriate size. Computations
with larger values of $x$ seem to indicate that the fraction of $1$s does tend 
toward $1/2$ as $x$ increases to infinity, but proving that this is the case 
seems far beyond the reach of existing techniques.

A much weaker observation is that as $x$ goes to infinity, the number of $1$s
in the binary representation of $3^x$ tends to infinity as well; that is, one 
would certainly be tempted to guess that there are only finitely many $x$ such 
that the binary representation of $3^x$ contains fewer than ten $1$s, or a 
hundred $1$s, or any given finite number of $1$s. This observation \emph{is} in
fact true, and was proven by Senge and Straus in 1973; their 
result~\cite[Theorem~3, p.~100]{SengeStraus1973} implies that for any given~$n$, 
there are only finitely many $x$ such that the binary representation of $3^x$
has $n$ or fewer bits equal to~$1$. In 1980 Cameron Stewart proved an effective
version of this result~\cite[Theorem~1, p.~64]{Stewart1980} --- which means that
given a value of~$n$, Stewart's arguments produce a bound $B(n)$ so that if 
$x>B(n)$, then $3^x$ has more than $n$ bits equal to~$1$. Unfortunately, the
values of $B(n)$ produced by Stewart's method grow very quickly; for example, 
we can show\footnote{
     Stewart's Theorem~1 shows that the largest $x$ for which $3^x$ has at 
     most $22$ bits equal to $1$ satisfies 
     $23 > (\log\log 3^x)/(C + \log\log\log 3^x)$ for some positive
     constant~$C$. We only get a stronger upper bound on $x$ if we solve
     for $x$ when $C=0$, and this is how we get our lower bound for $B(22)$.}
that $B(22) > 4.9\!\times\!10^{46}$.

In this paper, we use completely elementary techniques to find all powers of $3$
whose binary representations have at most twenty-two bits equal to~$1$. In fact,
these powers of $3$ are exactly the ones displayed in Table~\ref{T:binary}.

\begin{theorem}
\label{T:POWERSOF2}
The only powers of \,$3$ that can be written as the sum of twenty-two or fewer 
distinct powers of \,$2$ are $3^x$, where $0\le x \le 25$.
\end{theorem}

In other words, there are more than twenty-two $1$s in the binary representation
of $3^x$ exactly when $x > 25$. Clearly, this bound is much smaller than the one
obtained from Stewart's theorem!

We also look at the complementary problem of finding powers of $2$ whose 
base-$3$ representations contain no $2$s and at most twenty-five $1$s. Stewart's
theorem applies here as well, and says that if $2^x$ can be expressed in this 
manner, then $x$ is less than a computable bound that is larger than 
$5.4\!\times\!10^{54}$. Our result shows that in fact $x\le 8$.

\begin{theorem}
\label{T:POWERSOF3}
The only powers of \,$2$ that can be written as the sum of twenty-five or fewer
distinct powers of \,$3$ are\textup{:}
\begin{align*}
  2^0 &= 3^0\\
  2^2 &= 3^0 + 3^1\\
  2^8 &= 3^0 + 3^1 + 3^2 + 3^5.
\end{align*}
\end{theorem}

Put differently, if $x\not\in\{0, 2, 8\}$ then the base-$3$ representation of
$2^x$ will contain either at least one~$2$, or at least twenty-six $1$s. This 
provides a tiny bit of confirmation for a conjecture of  
Erd\H{o}s~\cite[Problem~1, p.~67]{Erdos1979}, which states that the only powers 
of $2$ whose base-$3$ representations contain only $0$s and $1$s are the three
examples given in Theorem~\ref{T:POWERSOF3}. (For work on Erd\H{o}s's conjecture
and closely related problems, see for example~\cite{
BennettBugeaudEtAl2013,
DupuyWeirich2016,
Lagarias2009,
Narkiewicz1980}
and the papers these articles cite.)

Theorems~\ref{T:POWERSOF2} and~\ref{T:POWERSOF3} can be expressed in terms of
exponential Diophantine equations. In particular, Theorem~\ref{T:POWERSOF2} 
gives us all solutions of 
\begin{equation}
\label{EQ:sumof2s}
3^x = 2^{a_1} + \cdots + 2^{a_n}, \qquad x\ge 0, \quad 0 \le a_1 < \cdots < a_n 
\end{equation}
for $n\le 22$, and Theorem~\ref{T:POWERSOF3} gives us all solutions to 
\begin{equation}
\label{EQ:sumof3s}
2^x = 3^{a_1} + \cdots + 3^{a_n}, \qquad x\ge 0, \quad 0 \le a_1 < \cdots < a_n 
\end{equation}
for $n\le 25$.

Our method for solving equations~\eqref{EQ:sumof2s} and~\eqref{EQ:sumof3s}
involves considering the equations modulo~$M$ for a sequence of well-chosen
moduli~$M$, each one dividing the next. We will postpone our discussion of what 
``well-chosen'' means, and for now we will simply illustrate our method with an 
example.

Let us look at the case $n=3$ of equation~\eqref{EQ:sumof2s}. We start by 
considering the related problem of writing a power of $3$ as the sum of three
powers of $2$ in the finite ring $\ZZ/M_1\ZZ$ for 
$M_1 = 5440 = 2^6 \cdot 5 \cdot 17$, where we no longer insist that the powers 
of~$2$ be distinct. The following diagram enumerates the powers of $2$ 
modulo~$M_1$; here the arrows indicate multiplication by~$2$.
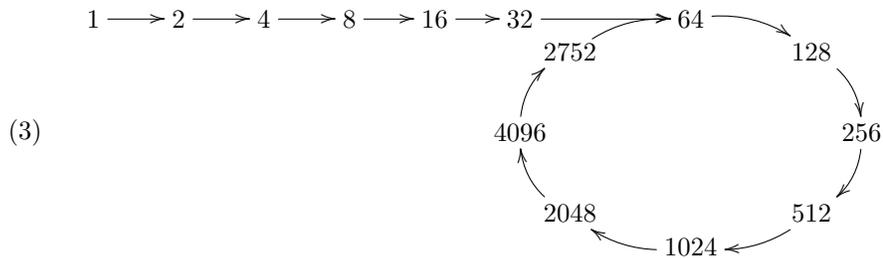
\begin{equation}
\label{EQ:2diagram}
\xygraph{
!{<0ex,0ex>; <15ex,0ex>: <0ex,10ex>::}
!{(-3.5,1)}*+{ 1}="1"
!{(-3  ,1)}*+{ 2}="2"
!{(-2.5,1)}*+{ 4}="4"
!{(-2  ,1)}*+{ 8}="8"
!{(-1.5,1)}*+{16}="16"
!{(-1  ,1)}*+{32}="32"
!{( 0  ,1)}*+{64}="N"
!{(-0.3         , 1           )}*+{}="LN"
!{( \halfroottwo, \halfroottwo)}*+{128}="NE"
!{( 1           , 0           )}*+{256}="E"
!{( \halfroottwo,-\halfroottwo)}*+{512}="SE"
!{( 0           ,-1           )}*+{1024}="S"
!{(-\halfroottwo,-\halfroottwo)}*+{2048}="SW"
!{(-1           , 0           )}*+{4096}="W"
!{(-\halfroottwo, \halfroottwo)}*+{2752}="NW"
"1"-@{>}"2"
"2"-@{>}"4"
"4"-@{>}"8"
"8"-@{>}"16"
"16"-@{>}"32"
"32"-@{>}"N"
"N"-@/^/@{>}"NE"
"NE"-@/^/@{>}"E"
"E"-@/^/@{>}"SE"
"SE"-@/^/@{>}"S"
"S"-@/^/@{>}"SW"
"SW"-@/^/@{>}"W"
"W"-@/^/@{>}"NW"
"NW"-@<-0.27ex>@/^/@{>}"N"
}
\end{equation}
We see there are $14$ distinct powers of~$2$ modulo~$M_1$, and likewise we find 
that there are $16$ distinct powers of~$3$. Using a computer to enumerate sums
of three powers of $2$ in $\ZZ/M_1\ZZ$, we find that (up to the order of the 
summands) there are only three ways to write a power of $3$ in $\ZZ/M_1\ZZ$ as a
sum of three powers of~$2$:
\begin{align}
\label{EQ:s1} 3^1 &\equiv  2^0 + 2^0 + 2^0 \bmod M_1\\
\label{EQ:s2} 3^2 &\equiv  2^0 + 2^2 + 2^2 \bmod M_1\\
\label{EQ:s3} 3^4 &\equiv  2^0 + 2^4 + 2^6 \bmod M_1.
\end{align}

For each of the summands $2^i$ on the right-hand side of one of these equations,
we can ask for the exponents $b$ such that $2^b \equiv 2^i \bmod M_1$. Looking
at diagram~\eqref{EQ:2diagram}, we see that for $i = 0, 2$, and~$4$, the only
exponent $b$ with $2^b \equiv 2^i\bmod M_1$ is $i$ itself, because $1$, $4$, and
$16$ are all on the ``tail'' of the diagram. On the other hand, the exponents
$b$ with $2^b \equiv 2^6\bmod M_1$ are 
$\{6, 14, 22, 30, \ldots \} = {\{6 + 8 j : j \ge 0\}}$, because the ``loop''
part of diagram~\eqref{EQ:2diagram} goes around in a cycle of $8$ steps.

Every solution to equation~\eqref{EQ:sumof2s} with $n=3$ must reduce 
modulo~$M_1$ to one of the three equations~\eqref{EQ:s1}, \eqref{EQ:s2}, 
or~\eqref{EQ:s3}. However, no solution to equation~\eqref{EQ:sumof2s} can reduce
to~\eqref{EQ:s1}, because the summands in~\eqref{EQ:sumof2s} would have to be
$2^0$, $2^0$, and $2^0$, which are not distinct. Likewise, no solution to 
equation~\eqref{EQ:sumof2s} can reduce modulo $M_1$ to~\eqref{EQ:s2}, because
two of the summands in~\eqref{EQ:sumof2s} would have to be~$2^2$. Therefore, 
every solution to equation~\eqref{EQ:sumof2s} with $n=3$ reduces modulo $M_1$
to~\eqref{EQ:s3}, and we see that two of the summands in~\eqref{EQ:sumof2s} must
be $2^0$ and~$2^4$.

Now we consider information modulo $M_2 = 2^7 \cdot 5 \cdot 17 \cdot 257$. If a
solution to equation~\eqref{EQ:sumof2s} reduces modulo $M_1$ to~\eqref{EQ:s3},
what can it reduce to modulo $M_2$? There are $16$ powers of $3$ in $\ZZ/M_2\ZZ$
that reduce to $3^4$ in $\ZZ/M_1\ZZ$, namely 
$3^4, 3^{4+16}, \ldots, 3^{4 + 15\cdot16}$, and there are $3$ powers of $2$ in 
$\ZZ/M_2\ZZ$ that reduce to $2^6$ in $\ZZ/M_1\ZZ$, namely $2^6$, $2^{14}$, 
and~$2^{22}$. We check that in $\ZZ/M_2\ZZ$ neither $2^0 + 2^4 + 2^{14}$ nor 
$2^0 + 2^4 + 2^{22}$ is equal to any of the possible powers of~$3$. However, 
$3^4 \equiv 2^0 + 2^4 + 2^6$ in $\ZZ/M_2\ZZ$.

Therefore, every solution to equation~\eqref{EQ:sumof2s} with $n=3$ must reduce 
modulo $M_2$ to the congruence $3^4 \equiv 2^0 + 2^4 + 2^6 \bmod M_2$. But we 
check that $2^0$, $2^4$, and $2^6$ lie on the tail of the analog of
diagram~\eqref{EQ:2diagram} for $M_2$, so the only powers of $2$ in the integers
that reduce to $2^0$, $2^4$, and $2^6$ modulo~$M_2$ are $2^0$, $2^4$, and $2^6$
themselves. We see that if there is a solution to equation~\eqref{EQ:sumof2s}
with $n=3$, the right-hand side must be $2^0 + 2^4 + 2^6$. As it happens, in 
the integers this sum is equal to $3^4$, so $3^4 = 2^0 + 2^4 + 2^6$ is the 
unique solution to equation~\eqref{EQ:sumof2s} with $n=3$.

This simple example displays the basic idea that we use to prove 
Theorem~\ref{T:POWERSOF2}. For such a small example we could have \emph{started}
by considering the equation modulo~$M_2$, instead of first looking modulo~$M_1$,
but for larger examples it is much more efficient to cut down the solution space
by looking first at small moduli before building up to larger ones.

Solving exponential Diophantine equations using congruence arguments is not a
new technique. In 1976, for example, Alex~\cite{Alex1976} used congruences to
find all solutions to $x + y = z$, where $x$, $y$, and $z$ are mutually coprime
integers divisible by no prime larger than~$7$. In 1982, Brenner and
Foster~\cite{BrennerFoster1982} presented a whole bestiary of exponential
Diophantine equations that can be solved in this way. (They mention in 
particular that Alex found all solutions to our example 
$3^x = 2^{a_1} + 2^{a_2} + 2^{a_3}$ using ``a few small moduli,'' although this
had been solved earlier by Pillai, as we discuss below.) In 2009, \'Ad\'am, 
Hajdu, and Luca~\cite{AdamHajduLuca2009} used a result of Erd\H{o}s, Pomerance,
and Schmutz~\cite{ErdosPomeranceSchmutz1991} to show that for every finite set
$S$ of primes and finite set $A\subset\ZZ$ of coefficients, the number of 
integers less than $x$ that can be written as the sum of a fixed number of terms
of the form $a s$, where $a\in A$ and $s\in\ZZ$ is a product of powers of primes
in~$S$, grows more slowly than a specific power of $\log x$. Independently, in a
2011 paper~\cite{DimitrovHowe2011} we studied representations of integers as 
sums of terms of the form $\pm 2^a 3^b$, which is the case $A = \{\pm 1\}$, 
$S = \{2,3\}$ of the problem studied in~\cite{AdamHajduLuca2009}. We presented
one way of finding moduli $M$ that could be used to prove that certain integers
cannot be represented by a given number of such terms, and we used the same
result of Erd\H{o}s, Pomerance, and Schmutz to show that there is a positive 
constant $c$ such that infinitely many integers $n$ cannot be written as a sum of
fewer than $c \log n / (\log\log n \log\log\log n)$ such terms.

In 2016 Bert\'ok and Hajdu~\cite{BertokHajdu2016} studied exponential
Diophantine equations in general, again using arguments based 
on~\cite{ErdosPomeranceSchmutz1991}, and they conjectured that if an exponential
Diophantine equation has only a finite number of solutions\footnote{
     The statement of the conjecture~\cite[p.~849]{BertokHajdu2016} only
     applies to Diophantine equations with \emph{no} solutions, but later in 
     the paper the authors show how the conjecture, if true, can be applied to
     equations that have finitely many solutions.}
and satisfies some other natural restrictions, then there is an integer $M$ such
that the solutions to the equation modulo~$M$ lift uniquely to the solutions 
in~$\ZZ$. In a later paper~\cite{BertokHajdu2018} the same authors generalized
this conjecture to number fields. One can view our work in this paper as
providing evidence in support of the Bert\'ok--Hajdu conjectures.

Our main contribution in this paper is the method we describe for choosing a 
sequence of moduli that allows us to refine the collection of solutions 
modulo~$M$, for larger and larger~$M$, until every solution modulo $M$ can be
lifted to at most one solution in the integers. Our moduli are chosen in a 
careful order that makes each refinement step computationally feasible. The 
closest predecessor to our technique seems to be the method used by Bert\'ok and
Hajdu in~\cite{BertokHajdu2016}, in which they choose a modulus $M$ and then 
piece together information gleaned from solutions to the original Diophantine
equation modulo the prime power divisors of $M$. Another new observation in this
paper appears in Section~\ref{S:extra}, where we show that any modulus $M$ that 
provides us with all solutions to equation~\eqref{EQ:sumof2s}
or~\eqref{EQ:sumof3s} must satisfy an unexpected condition.

We study the problem of writing powers of $2$ as sums of distinct powers of~$3$,
as well as the complementary problem of writing powers of $3$ as sums of 
distinct powers of~$2$, for several reasons. First, these problems are
simply-stated and natural. Second, we wanted to see what we could say about
Erd\H{o}s's conjecture. Third, we were curious how far the modular methods 
discussed by Brenner and Foster can be pushed, since even modest laptop 
computers are much more powerful than anything available at the time their paper
was written. And finally, we hope to bring these straightforward modular 
techniques to the attention of the community of mathematicians who are 
interested in exponential Diophantine equations.

As a historical note, we observe that the solutions to the case $n=2$ of
equations~\eqref{EQ:sumof2s} and~\eqref{EQ:sumof3s} were determined nearly
seven centuries ago by Levi ben Gerson~\cite{benGerson1342}, who showed that 
the only pairs of integers of the form $2^r 3^s$ that differ by $1$ are $(1,2)$,
$(2,3)$, $(3,4)$, and $(8,9)$. A paraphrase of ben Gerson's argument, more 
legible\footnote{
  The adjective is chosen with intention. Follow the link in the bibliography
  to understand why.}
than~\cite{benGerson1342}, is given 
in~\cite[Appendice, pp.~183--191]{ChemlaPahaut1992}. One way to prove ben 
Gerson's theorem is to observe that every solution to ben Gerson's problem is a
solution to the case $n=2$ of either equation~\eqref{EQ:sumof2s} or 
equation~\eqref{EQ:sumof3s}, and then to consider those two equations 
modulo~$80$.

In 1945, Pillai~\cite{Pillai1945} found all solutions to 
$\pm(2^x - 3^y) = 2^X + 3^Y$; taking either $x$ or $y$ to be $0$ leads to the 
solutions for the case $n=3$ of equations~\eqref{EQ:sumof2s} 
and~\eqref{EQ:sumof3s}. Between 2011 and 2013, Bennett, Bugeaud, and 
Mignotte~\cite{BennettBugeaudEtAl2013,BennettBugeaudEtAl2012} used linear forms
in two logarithms to find all perfect powers whose binary representations have 
at most four bits equal to $1$ (extending a result of Szalay~\cite{Szalay2002}
that gives all perfect squares with at most three bits equal to~$1$), and this 
solves the case $n=4$ of equation~\eqref{EQ:sumof2s}. These are all of the 
previous solutions to cases of equations~\eqref{EQ:sumof2s} 
and~\eqref{EQ:sumof3s} that we are aware of; however, the paper of Bert\'ok and 
Hajdu~\cite{BertokHajdu2016} discussed earlier includes solutions to many very 
similar equations, including, for example, finding all powers of $17$ that can
be expressed as the sum of nine distinct powers of~$5$. Surely their methods 
could have been used to solve some more instances of 
equations~\eqref{EQ:sumof2s} and~\eqref{EQ:sumof3s}.

The structure of this paper is as follows: In Section~\ref{S:prelim} we briefly
review some notation. In Section~\ref{S:extra} we observe that in some 
situations there will necessarily be solutions to equations~\eqref{EQ:sumof2s}
or~\eqref{EQ:sumof3s} modulo~$M$ that are not reductions of solutions in the
integers, unless some specific conditions on $M$ hold. These conditions shape 
our strategy of choosing a specific sequence of moduli to use in the proofs of 
Theorems~\ref{T:POWERSOF2} and~\ref{T:POWERSOF3}. In Section~\ref{S:lifting} we
give examples of two different ways of lifting solutions to~\eqref{EQ:sumof2s}
modulo $M_1$ to solutions modulo~$M_2$, suitable for two different
circumstances. These examples help clarify the process by which we proved
Theorems~\ref{T:POWERSOF2} and~\ref{T:POWERSOF3}. We present the proofs of these
theorem in Sections~\ref{S:2} and~\ref{S:3}.

The programs we used to complete our calculations were written in 
Magma~\cite{BosmaCannonEtAl1997} and are available as supplementary material
attached to the ArXiv version of this paper~\cite{DimitrovHowe2021}. They are
also available on the second author's web site.

\subsection*{Acknowledgments}
We are grateful to Lajos Hajdu for his comments on an earlier version of this
paper, and to the anonymous referees for their helpful suggestions.

\section{Notation and conventions}
\label{S:prelim}

In this paper we will often want to count or enumerate the number of solutions
to an exponential Diophantine equation modulo~$M$, but there is some natural 
ambiguity as to what this might mean. For instance, there are infinitely many 
pairs of integers $x\ge 0$ and $y\ge 0$ for which the congruence 
$3^x \equiv 2^y + 5\bmod 28$ holds, but for every such $x$ and $y$ we have
$3^x \equiv 9\bmod 28$ and $2^y\equiv 1\bmod 28$, so it might not be
unreasonable to say that there is only one solution. In order to avoid any
confusion, we remove this ambiguity by adopting the following convention.

\begin{convention}
\label{convention}
When we count or enumerate solutions to an exponential Diophantine equation
modulo~$M$, we will consider two solutions to be the same if the corresponding
terms in the equation are congruent modulo $M$.
\end{convention}

This means, for example, that for the congruence $3^x \equiv 2^y + 5\bmod 28$ we
consider the solutions $(x,y) = (2,2)$, $(x,y) = (8,2)$, and $(x,y) = (8,5)$ to 
be the same, because in each case $3^x\equiv 9 \bmod 28$ and 
$2^y\equiv 4\bmod 28$.

This convention does have one drawback, which is that for some exponential
Diophantine equation modulo~$M$, there truly are only finitely many integer 
solutions. For example, the only integers $x\ge 0$ and $y\ge 0$ such that
$3^x \equiv 2^y + 5\bmod 216$ are $x=2$ and $y=2$. This distinction will in fact
be important to us, so we make the following definition.

\begin{definition}
Let $M>0$ be an integer and $p$ a prime. We say that a power of~$p$, say $p^i$, 
is \emph{determinate} modulo $M$ if the only integer $b\ge 0$ with 
$p^b \equiv p^i\bmod M$  is $b=i$\textup{;}  otherwise, we say that $p^i$ is an
\emph{indeterminate} power of~$p$ modulo~$M$.
\end{definition}

Thus, we will say that the congruence $3^x\equiv 2^y + 5\bmod 28$ has one 
solution, namely $3^2\equiv 2^2 + 5\bmod 28$, but that $3^2$ is an indeterminate
power of $3$ modulo $28$ and $2^2$ is an indeterminate power of $2$ modulo~$28$.
On the other hand, ${3^x\equiv 2^y + 5\bmod 216}$ also has only one solution, but 
the power of $3$ and the power of $2$ involved are both determinate.

Given a prime $p$ and an integer $M>0$, we can construct a diagram like
diagram~\eqref{EQ:2diagram} of the powers of $p$ modulo~$M$. Note that a 
determinate power of $p$ modulo~$M$ is exactly a power of $p$ that lies on the 
tail of this diagram, and a straightforward argument shows that for $i\ge 0$, the
integer $p^i$ is a determinate power of $p$ modulo~$M$ if and only if $M$ is
divisible by~$p^{i+1}$. 

Recall that if $M$ is a positive integer then the group of units in the ring 
$\ZZ/M\ZZ$ has order $\varphi(M)$, where $\varphi$ is the Euler 
$\varphi$-function, which can be computed using the formula 
$\varphi(n) = n \prod_{p\mid n} (1 - 1/p)$; 
see~\cite[\S2.3, \S2.5]{Apostol1976}. Also, if $M$ is an odd prime power then
the group of units in $\ZZ/M\ZZ$ is 
cyclic~\cite[Theorem~10.4, p.~207]{Apostol1976}.
            
For every prime $p$, we let $v_p$ be the \emph{$p$-adic valuation} function, so
that $v_p(M)$ is the largest $x$ such that $p^x$ divides $M$. And lastly, we set
some notation related to the behavior of the numbers $2$ and $3$ in finite 
rings. 

\begin{notation}
\label{N:23}
Let $M$ be a positive integer and write $M = 2^u 3^v M'$, where $u = v_2(M)$ and
$v = v_3(M)$, so that $M'$ is coprime to $6$.
\begin{itemize}
\item We let $O_2(M)$ be the multiplicative order of $2$ in the
      ring $\ZZ/3^v M'\ZZ$.
\item We let $O_2'(M)$ be the multiplicative order of $2$ in the
      ring $\ZZ/M'\ZZ$.
\item We let $O_3(M)$ be the multiplicative order of $3$ in the
      ring $\ZZ/2^u M'\ZZ$.
\item We let $O_3'(M)$ be the multiplicative order of $3$ in the
      ring $\ZZ/M'\ZZ$.
\end{itemize}
\end{notation}
We see, for example, that there are $v_2(M)+ O_2(M)$ elements in the
tail-and-loop diagram of the powers of $2$ modulo~$M$, with $v_2(M)$ in the
tail and $O_2(M)$ in the loop. Similarly, there are $v_3(M)+ O_3(M)$ elements
in the tail-and-loop diagram of the powers of $3$ modulo~$M$.

\section{Extraneous solutions to congruences}
\label{S:extra}

The basic heuristic behind our strategy for solving instances of 
equations~\eqref{EQ:sumof2s} and~\eqref{EQ:sumof3s} is that if $M$ is large and 
there are very few powers of $2$ in $\ZZ/M\ZZ$ and very few powers of $3$ in 
$\ZZ/M\ZZ$, then there should be very few ``extraneous'' solutions to 
equations~\eqref{EQ:sumof2s} or~\eqref{EQ:sumof3s} modulo $M$ --- that is,
solutions that are not the reduction modulo $M$ of a solution in the integers.
If $M$ is divisible by sufficiently high powers of $2$ and/or~$3$, we can hope
that every solution modulo $M$ to equation~\eqref{EQ:sumof2s}
or~\eqref{EQ:sumof3s} will involve only determinate powers of $2$ or of $3$ 
modulo~$M$ (where \emph{determinate} is as defined in Section~\ref{S:prelim}).
If this is the case, then each solution will lift uniquely to the integers, if
it lifts at all. However, it turns out that for many moduli~$M$, if there is
\emph{any} solution to one of these equations, then there is \emph{also} a 
solution that includes indeterminate powers of $2$ and of~$3$.

For example, we saw in the introduction that if 
$M_1 = 5440 = 2^6\cdot 5\cdot 17$ then the equation
$3^x \equiv 2^{a_1} + 2^{a_2} + 2^{a_3} \bmod M_1$ has the three solutions
given by~\eqref{EQ:s1}, \eqref{EQ:s2}, and \eqref{EQ:s3}, and we see that 
\eqref{EQ:s3} involves an indeterminate power of $2$ (and of~$3$). If we look at
the same equation modulo $M_2$, where $M_2 = 2 M_1 = 2^7\cdot 5\cdot 17$, then 
we find four solutions, including $3^{20} \equiv 2^0 + 2^4 + 2^{14}$, and this 
involves indeterminate powers of $2$ and of $3$ modulo~$M_2$. When we look at
the same equation modulo~$M_3$, where 
$M_3 = 41 M_2 = 2^7\cdot 5\cdot 17 \cdot 41$, there is once again a solution 
with indeterminate powers of $2$ and $3$, namely 
$3^{20} \equiv 2^0 + 2^4 + 2^{46}$. And the same happens yet again when we work
modulo~$M_4$, where $M_4 = 193 M_3 = 2^7\cdot 5\cdot 17 \cdot 41 \cdot 193$.

And yet in the introduction, when we considered solutions to 
$3^x \equiv 2^{a_1} + 2^{a_2} + 2^{a_3}$ modulo $2^7 \cdot 5 \cdot 17\cdot 257$,
we did \emph{not} wind up with extraneous solutions. What is the difference
between $2^7 \cdot 5 \cdot 17\cdot 257$ and 
$2^7\cdot 5\cdot 17 \cdot 41 \cdot 193$?

The following proposition, which uses Notation~\ref{N:23}, explains one way in
which solutions with indeterminate powers of $2$ or $3$ can arise, and suggests
a condition that we will want to impose on the moduli we use.

\begin{lemma}
\label{L:newsolution}
Let $M$ be a positive integer. Suppose $x>2$, $y>0$, and $c$ are integers such
that $3^y \equiv c + 2^x \bmod M$. If $O_3'(M)$ is not divisible by $2^{x-1}$
and $O_2'(M)$ is not divisible by~$3^{y}$, then there are integers $x'\ge 0$ and
$y'\ge 0$ such that 
\begin{enumerate}
\item[(a)] $3^{y'} \equiv c + 2^{x'}\bmod M$,
\item[(b)] $2^{x'}$ is an indeterminate power of $2$ modulo $M$, and
\item[(c)] $3^{y'}$ is an indeterminate power of $3$ modulo $M$.
\end{enumerate}
\end{lemma}

Lemma~\ref{L:newsolution} shows that in the example we presented in the 
introduction, it was necessary for us to use a modulus divisible by a prime 
(in our case,~$257$) for which either the order of $3$ is divisible by $2^5$
or the order of $2$ is divisible by~$3^4$. Since ${3^4 = 2^0 + 2^4 + 2^6}$, if
we use a modulus $M$ that is divisible by $2^7$ (so that $2^0$, $2^4$, and $2^6$
are determinate powers of $2$ modulo~$M$), Lemma~\ref{L:newsolution} shows 
that there will be other, extraneous, solutions modulo $M$ unless $M$ is 
divisible by such a prime.

\begin{proof}[Proof of Lemma~\textup{\ref{L:newsolution}}]
Write $M = 2^u 3^v M'$ where $M'$ is an integer coprime to $6$, and set
$o_2 = O_2'(M)$ and $o_3 = O_3'(M)$. First we claim that there is an integer $s$
such that $y + so_3 > v$ and $3^{y + s o_3} \equiv c \bmod 2^u$. 

Suppose $u\le x$, so that $3^y\equiv c\bmod 2^u$. We know that 
$3^s\equiv 1\bmod 2^u$ if $s$ is a multiple of $\varphi(2^u)$, so we can simply 
take $s$ to be a large enough multiple $\varphi(2^u)$ so that $y + so_3 > v$,
and this $s$ meets the conditions of our claim.

Suppose $u > x$. Then $M$ is even, and since $c$ differs from $3^y$ by a 
multiple of the even number~$M$, we see that $c$ must be odd. Therefore there is
an integer $d$ such that $cd\equiv 1 \bmod 2^u$. Choose such a $d$ and consider
the integer $z = 1 + 2^x d$, which is congruent to $1\bmod 8$ because $x>2$. If
we apply part~\ref{logbase3mod2} of Lemma~\ref{L:logs} (below) to this~$z$, we
find that there is an integer $e_0$, divisible by $2^{x-2}$, such that every
integer $e$ with $e\equiv e_0\bmod 2^{u-2}$ satisfies 
$3^e \equiv 1 + 2^x d \bmod 2^u$. By assumption, the highest power of $2$ that
divides $o_3$ is at most $2^{x-2}$. Therefore there is an integer $s$ such that
$s o_3\equiv - e_0\bmod 2^{u-2}$, and we can choose such an $s$ that is large
enough so that $y + so_3 > v$.

We have $3^{-s o_3} \equiv 1 + 2^x d \bmod 2^u$. Multiplying both sides of this
congruence by $c\,3^{s o_3}$ gives ${c \equiv (c + 2^x) 3^{s o_3} \bmod 2^u}$, and
since $c + 2^x \equiv 3^y \bmod M$ and hence also modulo $2^u$, we find that
$c \equiv 3^{y + s o_3}\bmod 2^u$. Thus, this $s$ has the properties we desire,
and we have proven our claim.

Similarly, using part~\ref{logbase2mod3} of Lemma~\ref{L:logs}, we can show that
there is an integer $r$ such that $x + r o_2 > u$ and 
$2^{x + r o_2} \equiv - c \bmod 3^v$.

Let $x' = x + r o_2$ and let $y' = y + s o_3$. We claim that this $x'$ and $y'$
satisfy conditions (a), (b), and (c) from the lemma. It is easy to check 
conditions (b) and (c) because $x' > u$ and $y' > v$ by construction. To check
condition~(a), we use the Chinese Remainder Theorem: It suffices to check that
$3^{y'} \equiv c + 2^{x'}$ modulo~$M'$, modulo~$2^u$, and modulo~$3^v$.

We have $2^{o_2} \equiv 1\bmod M'$ and $3^{o_3}\equiv 1\bmod M'$ by the 
definitions of $o_2$ and $o_3$, so $3^{y'}\equiv 3^y\bmod M'$ and 
$2^{x'}\equiv 2^x\bmod M'$, and we have $3^{y'} \equiv c + 2^{x'}\bmod M'$.

We have $2^{x'}\equiv 0\bmod 2^u$ because $x + r o_2 > u$ by construction.
Since $3^{y'}  \equiv 3^{y + s o_3} \equiv c\bmod 2^u$, we have
$3^{y'} \equiv c + 2^{x'}\bmod 2^u$.

The same reasoning shows that we have $3^{y'}\equiv 0\bmod 3^v$, and since 
$2^{x'}\equiv 2^{x + r o_2} \equiv -c\bmod 3^v$, we have
${3^{y'} \equiv c + 2^{x'}\bmod 3^v}$. This shows that condition (a) holds for
this $x'$ and~$y'$, and completes the proof of the lemma.
\end{proof}

\begin{lemma}
\label{L:logs}
\quad
\begin{enumerate}
\item \label{logbase3mod2}
      Let $z$ be an integer with $z\equiv 1 \bmod 8$. For every integer $u\ge 3$
      there is an integer~$e_0$ such that the integers $e$ that satisfy
      $3^e \equiv z \bmod 2^u$ are precisely the integers $e$ that satisfy 
      $e\equiv e_0\bmod 2^{u-2}$. If $x\le u$ is an integer with 
      $z\equiv 1 \bmod 2^x$, then $e_0$ is divisible by $2^{x-2}$.
\item \label{logbase2mod3}
      Let $z$ be an integer with $z\equiv 1\bmod 3$. For every integer $v\ge 1$
      there is an integer~$e_0$ such that the integers $e$ that satisfy 
      $2^e \equiv z \bmod 3^v$ are precisely the integers $e$ that satisfy
      $e\equiv e_0\bmod 2\cdot 3^{v-1}$. If $y\le v$ is an integer with 
      $z\equiv 1 \bmod 3^y$, then $e_0$ is divisible by $2\cdot 3^{y-1}$.
\end{enumerate}
\end{lemma}

\begin{proof}
For statement~\ref{logbase3mod2}: We leave the reader to show that for every
$u\ge 3$, the order of $3$ modulo $2^u$ is~$2^{u-2}$. (The proof can be modeled 
after the proof of~\cite[Theorem~10.11, p.~218]{Apostol1976}.) Since there are 
$2^{u-1}$ units in $\ZZ/2^u\ZZ$, and the order of $3$ is half of this, it 
follows that half of the units are powers of~$3$. A power of $3$ is never 
congruent to $5$ or $7$ modulo~$8$, and this accounts for half of the units. 
Therefore, every unit that is $1$ or $3$ modulo~$8$ is a power of~$3$. Thus, 
there is an $e_0$ such that $3^{e_0} \equiv z$. The fact that 
$3^e \equiv z \bmod 2^u$ if and only if $e \equiv e_0\bmod 2^{u-2}$ is simply a
consequence of the fact that the order of $3$ modulo $2^u$ is $2^{u-2}$.

If $z\equiv 1 \bmod 2^x$ with $x\le u$, then $3^{e_0} \equiv 1 \bmod 2^x$, so
$e_0$ is a multiple of the order of $3$ modulo $2^x$, and hence $e_0$ is
divisible by~$2^{x-2}$.

The proof of statement~\ref{logbase2mod3} is analogous, and we leave it to the
reader.
\end{proof}

When we look at cases of equation~\eqref{EQ:sumof2s} with larger values of~$n$, 
we will find that Lemma~\ref{L:newsolution} tells us that we will need to
include information gleaned from moduli divisible by primes $p$ such that the
order of $3$ modulo $p$ is divisible by quite large powers of~$2$. In 
Section~\ref{S:2} we show how we can work our way up to such moduli.

\section{Lifting solutions}
\label{S:lifting}

Our proofs of Theorems~\ref{T:POWERSOF2} and~\ref{T:POWERSOF3} are 
computational. In each proof, we consider a sequence of moduli 
$M_1, M_2, \ldots$, each dividing the next. Roughly speaking, we first compute
the solutions to equation~\eqref{EQ:sumof2s} or~\eqref{EQ:sumof3s} modulo~$M_1$;
then for each $i>1$ in turn we ``lift'' the solutions modulo~$M_{i-1}$ to 
solutions modulo~$M_i$. We stop when we have reached an $M_i$ where all of the
summands that appear on the right-hand side of the solutions modulo~$M_i$ are
determinate (in the sense defined in Section~\ref{S:prelim}); at that point, 
each solution modulo~$M_i$ can be lifted uniquely to a solution in the integers,
if it lifts to a solution at all.

This strategy depends on our having efficient methods for lifting a solution 
modulo~$M_{i-1}$ to a solution modulo~$M_i$. In Section~\ref{S:2} we will spell
out our methods more formally, but in this section we would like to give two 
examples to help make the methods more clear. For the sake of exposition, we 
will focus on finding solutions to equation~\eqref{EQ:sumof2s} modulo~$M$ for
various $M$, and as we did in the introduction, we will ignore the requirement
that the summands be distinct.

As an example of one extreme case of the lifting problem, let $M_1 = 439$ and
let $n=12$ and consider the following solution to equation~\eqref{EQ:sumof2s} 
modulo~$M_1$:
\begin{equation}
\label{EQ:liftexample1}
3^{57} \equiv 2^0 + 2^1 + 2^{11} + 2^{12} + 2^{15} + 2^{16} + 2^{26} 
                        + 2^{27} + 2^{37} + 2^{57} + 2^{65} + 2^{68}.
\end{equation}
Let $p$ be the prime $9361973132609$ and let $M_2 = p M_1$. We will try to find
a lift of the solution~\eqref{EQ:liftexample1} to a solution modulo~$M_2$. We
compute that the graph of the powers of $2$ modulo~$M_1$ forms a loop of cycle
length $73$ with no tail\ellipsis\ and we compute that the graph of powers of
$2$ modulo~$M_2$ is \emph{also} a tailless loop of cycle length~$73$. That means
that there is \emph{exactly one} power of $2$ in $\ZZ/M_2\ZZ$ that reduces to a
given power of $2$ in $\ZZ/M_1\ZZ$. If we can lift 
equation~\eqref{EQ:liftexample1} to a solution modulo~$M_2$, then the right-hand
side of the lifted solution will have to be
\[
2^0 + 2^1 + 2^{11} + 2^{12} + 2^{15} + 2^{16} + 2^{26} 
          + 2^{27} + 2^{37} + 2^{57} + 2^{65} + 2^{68} \bmod M_2.
\]
If we let $z$ be this sum, then to determine whether there is a lift of
equation~\eqref{EQ:liftexample1} to a solution modulo~$M_2$, we simply have to 
determine whether there is an $x$ such that $3^x \equiv z\bmod M_2$.

It turns out that the graph of powers of $3$ modulo~$M_2$ is a tailless loop 
with cycle length $p-1 = 9361973132608$, so we definitely do \emph{not} want to
find $x$ (if it exists) by enumeration. Instead, we can find $x$ by using 
discrete logarithms.

If there is an $x$ with $3^x \equiv z \bmod M_2$, then that same $x$ satisfies 
$3^x \equiv z \bmod p$ for the prime $p = M_2/M_1$. We can find an $x$ that 
satisfies this congruence if and only if $z\in(\ZZ/p\ZZ)^*$ lies in the subgroup
of $(\ZZ/p\ZZ)^*$ generated by~$3$. Using the computer algebra package Magma, we
find that in fact $3$ generates the whole group of units, and Magma very quickly
computes a discrete logarithm of $z$ with respect to~$3$ --- that is, an integer
$x$ with $3^x \equiv z \bmod p$. In fact, every integer $x$ satisfying
\begin{equation}
\label{EQ:logbasep}
x \equiv 3976447101915\bmod (p-1)
\end{equation}
will give a solution to this congruence.

In order for $x$ to give a solution modulo~$M_2$, we also need to have
$3^x \equiv z \bmod M_1$. The graph of powers of $3$ modulo~$M_1$ is a tailless
loop with cycle length $146$, and we find that for $x$ to solve this congruence 
modulo $M_1$ we need to have $x \equiv 57 \bmod 146$.

But $146$ is a divisor of $p-1$, and reducing equation~\eqref{EQ:logbasep} 
modulo~$146$, we find that it becomes $x \equiv 31 \bmod 146$. This is 
incompatible with the congruence from the preceding paragraph, so there is no
$x$ with $3^x \equiv z \bmod M_2$. This shows that 
equation~\eqref{EQ:liftexample1} cannot be lifted to a solution modulo~$M_2$.

Let us turn to another example, which demonstrates a different approach to the
lifting problem. We again take $M_1 = 439$ and start with the solution to 
equation~\eqref{EQ:sumof2s} modulo~$M_1$ given by~\eqref{EQ:liftexample1}. This
time, however, we take $p = 1753$ and $M_2 = p M_1$. We will try to find a lift
of the solution~\eqref{EQ:liftexample1} to a solution modulo~$M_2$.

The graph of powers of $2$ modulo~$M_2$ is a tailless loop of cycle 
length~$146$, which is exactly twice as long as the cycle of powers of $2$
modulo~$M_1$. That means that there are \emph{exactly two} powers of $2$ 
modulo~$M_2$ that reduce to a given power of $2$ modulo~$M_1$. In particular,
the two lifts to $\ZZ/M_2\ZZ$ of the element $2^i\in \ZZ/M_1\ZZ$ are $2^i$
and~$2^{i + 73}$.

Similarly, we can also compute that there are six lifts of $3^{57}\in\ZZ/M_1\ZZ$
to powers of $3$ in $\ZZ/M_2\ZZ$, namely $3^{57}$, $3^{203}$, $3^{349}$, 
$3^{495}$, $3^{641}$, and $3^{787}$.

We see that every summand on the right-hand side of~\eqref{EQ:liftexample1} has
two lifts to $\ZZ/M_2\ZZ$, and the left-hand side has six lifts. In principle,
we could compute all $6\cdot 2^{12} = 24{,}576$ lifts of the terms appearing
in~\eqref{EQ:liftexample1} and check to see which combinations of lifts give us
an equality modulo~$M_2$, but this would be inefficient\ellipsis\ and for larger
values of~$n$, it would become more and more inefficient. 

Instead, we use a ``meet in the middle'' technique. We rewrite 
equation~\eqref{EQ:liftexample1} to get the following congruence modulo~$M_1$:
\begin{equation}
\label{EQ:liftexample2}
3^{57} - 2^0 - 2^1 - 2^{11} - 2^{12} - 2^{15} 
\equiv
2^{16} + 2^{26} + 2^{27} + 2^{37} + 2^{57} + 2^{65} + 2^{68}.
\end{equation}
There are $6\cdot 2^5 = 192$ lifts to $\ZZ/M_2\ZZ$ of the terms appearing on the
left-hand side of~\eqref{EQ:liftexample2}, and $2^7 = 128$ lifts of the terms on
the right-hand side. We compute the values (modulo~$M_2$) of all of the 
left-hand lifts, and the values of all of the right-hand lifts, and then compare
the two lists to see whether there are any values in common. (We can quickly
find these common values if we sort each list first.) Each such common value $w$
gives us one (or more) lifts to $\ZZ/M_2\ZZ$ of~\eqref{EQ:liftexample2}, and 
hence also of~\eqref{EQ:liftexample1}. And clearly, all solutions 
to~\eqref{EQ:sumof2s} modulo~$M_2$ that are lifts of~\eqref{EQ:liftexample1}
will arise in this way. In point of fact, for this particular example we found 
eight values of~$w$, from which we obtained eight solutions
to~\eqref{EQ:sumof2s} in $\ZZ/M_2\ZZ$ that were lifts 
of~\eqref{EQ:liftexample1}.

The two techniques we have demonstrated here for lifting solutions 
of~\eqref{EQ:sumof2s} modulo~$M_1$ to solutions modulo~$M_2$ are the basis for
the procedure for proving Theorem~\ref{T:POWERSOF2} that we sketch in the 
following section.

\section{Proof of Theorem~\ref{T:POWERSOF2}}
\label{S:2}

To prove Theorem~\ref{T:POWERSOF2} we consider the sequence of moduli~$M_i$,
where $M_i = \prod_{j\le i}m_j$ for the factors $m_1$, \ldots, $m_{64}$ listed
in Table~\ref{T:2}, so that each $M_i$ divides the next. As we explained in 
Section~\ref{S:lifting}, roughly speaking we first compute the solutions to 
equation~\eqref{EQ:sumof2s} in $\ZZ/M_1\ZZ$; then, using the ideas sketched out 
in the examples in Section~\ref{S:lifting}, we lift the solutions to
$\ZZ/M_2\ZZ$, then to $\ZZ/M_3\ZZ$, then to $\ZZ/M_4\ZZ$, and so on, stopping
when we have reached an $M_i$ where all of the powers of~$2$ that appear in the
solutions are determinate. If all the powers of $2$ in a solution are 
determinate, the solution can be lifted uniquely to a solution in the integers,
if it lifts to a solution at all.

\afterpage{%
\clearpage
\begin{landscape}
\centering 
 
\begin{table}
\label{T:2}        
\caption{Data for the factors $m_i$ and the moduli $M_i = \prod_{j\le i} m_j$
         used in the proof of Theorem~\ref{T:POWERSOF2}. The notation in the
         table headings is as in Notation~\ref{N:23}.}
\footnotesize
\renewcommand{\arraystretch}{0.925} 
\begin{tabular}{r r r r r c c r r r r r c}
\toprule
 $i$ & $m_i$                              & $O_2(m_i)$         & $O_2(M_i)$         & $O_3'(m_i)$              & $v_2(O_3'(M_i))$ &\hbox to 0em{} & $i$  & $m_i$                    & $O_2(m_i)$         & $O_2(M_i)$         & $O_3'(m_i)$              & $v_2(O_3'(M_i))$ \\
\cmidrule{1-6}\cmidrule{8-13}
 $1$ & $2^{4} \cdot 7 \cdot 73$           & $             3^2$ & $             3^2$ & $     3 \cdot 2^{ 2\pz}$ & $\pz 2$          &               & $32$ & $113246209$              & $2^{20} \cdot 3^2$ & $2^{20} \cdot 3^2$ & $    27 \cdot 2^{19}$    & $20$ \\
 $2$ & $3^{3} \cdot 19$                   & $2      \cdot 3^2$ & $2      \cdot 3^2$ & $     9 \cdot 2^{ 1\pz}$ & $\pz 2$          &               & $33$ & $319489$                 & $2^{12} \cdot 3^0$ & $2^{20} \cdot 3^2$ & $    39 \cdot 2^{ 8\pz}$ & $20$ \\
 $3$ & $5 \cdot 13 \cdot 37 \cdot 109$    & $2^{ 2} \cdot 3^2$ & $2^{ 2} \cdot 3^2$ & $    27 \cdot 2^{ 2\pz}$ & $\pz 2$          &               & $34$ & $1084521185281$          & $2^{21} \cdot 3^2$ & $2^{21} \cdot 3^2$ & $ 43095 \cdot 2^{22}$    & $22$ \\
 $4$ & $241 \cdot 433$                    & $2^{ 3} \cdot 3^2$ & $2^{ 3} \cdot 3^2$ & $   135 \cdot 2^{ 3\pz}$ & $\pz 3$          &               & $35$ & $2^{2}$                  &          ---\pz    & $2^{21} \cdot 3^2$ &         ---\pz\pz        & $22$ \\
 $5$ & $17$                               & $2^{ 3} \cdot 3^0$ & $2^{ 3} \cdot 3^2$ & $             2^{ 4\pz}$ & $\pz 4$          &               & $36$ & $7348420609$             & $2^{22} \cdot 3^1$ & $2^{22} \cdot 3^2$ & $    73 \cdot 2^{24}$    & $24$ \\
 $6$ & $2^{2}$                            &         ---\pz     & $2^{ 3} \cdot 3^2$ &         ---\pz\pz        & $\pz 4$          &               & $37$ & $2^{2}$                  &         ---\pz     & $2^{22} \cdot 3^2$ &         ---\pz\pz        & $24$ \\
 $7$ & $38737$                            & $2^{ 3} \cdot 3^2$ & $2^{ 3} \cdot 3^2$ & $  2421 \cdot 2^{ 3\pz}$ & $\pz 4$          &               & $38$ & $448203325441$           & $2^{23} \cdot 3^1$ & $2^{23} \cdot 3^2$ & $ 26715 \cdot 2^{21}$    & $24$ \\
 $8$ & $97 \cdot 577$                     & $2^{ 4} \cdot 3^2$ & $2^{ 4} \cdot 3^2$ & $     3 \cdot 2^{ 4\pz}$ & $\pz 4$          &               & $39$ & $1107296257$             & $2^{24} \cdot 3^1$ & $2^{24} \cdot 3^2$ & $    11 \cdot 2^{22}$    & $24$ \\
 $9$ & $257 \cdot 673$                    & $2^{ 4} \cdot 3^1$ & $2^{ 4} \cdot 3^2$ & $    21 \cdot 2^{ 8\pz}$ & $\pz 8$          &               & $40$ & $167772161$              & $2^{24} \cdot 3^0$ & $2^{24} \cdot 3^2$ & $     5 \cdot 2^{25}$    & $25$ \\
$10$ & $2^{4}$                            &         ---\pz     & $2^{ 4} \cdot 3^2$ &         ---\pz\pz        & $\pz 8$          &               & $41$ & $2$                      &         ---\pz     & $2^{24} \cdot 3^2$ &         ---\pz\pz        & $25$ \\
$11$ & $193 \cdot 1153$                   & $2^{ 5} \cdot 3^2$ & $2^{ 5} \cdot 3^2$ & $     9 \cdot 2^{ 6\pz}$ & $\pz 8$          &               & $42$ & $74490839041$            & $2^{26} \cdot 3^1$ & $2^{26} \cdot 3^2$ & $   185 \cdot 2^{26}$    & $26$ \\
$12$ & $6337$                             & $2^{ 5} \cdot 3^2$ & $2^{ 5} \cdot 3^2$ & $    99 \cdot 2^{ 4\pz}$ & $\pz 8$          &               & $43$ & $2$                      &         ---\pz     & $2^{26} \cdot 3^2$ &         ---\pz\pz        & $26$ \\
$13$ & $65537$                            & $2^{ 5} \cdot 3^0$ & $2^{ 5} \cdot 3^2$ & $             2^{16}$    & $   16$          &               & $44$ & $246423748609$           & $2^{26} \cdot 3^1$ & $2^{26} \cdot 3^2$ & $    27 \cdot 2^{28}$    & $28$ \\
$14$ & $2^{8}$                            &         ---\pz     & $2^{ 5} \cdot 3^2$ &         ---\pz\pz        & $   16$          &               & $45$ & $2^{2}$                  &         ---\pz     & $2^{26} \cdot 3^2$ &         ---\pz\pz        & $28$ \\
$15$ & $641$                              & $2^{ 6} \cdot 3^0$ & $2^{ 6} \cdot 3^2$ & $     5 \cdot 2^{ 7\pz}$ & $   16$          &               & $46$ & $29796335617$            & $2^{27} \cdot 3^1$ & $2^{27} \cdot 3^2$ & $   111 \cdot 2^{24}$    & $28$ \\
$16$ & $769$                              & $2^{ 7} \cdot 3^1$ & $2^{ 7} \cdot 3^2$ & $     3 \cdot 2^{ 4\pz}$ & $   16$          &               & $47$ & $3221225473$             & $2^{28} \cdot 3^1$ & $2^{28} \cdot 3^2$ & $             2^{27}$    & $28$ \\
$17$ & $274177$                           & $2^{ 7} \cdot 3^0$ & $2^{ 7} \cdot 3^2$ & $   153 \cdot 2^{ 5\pz}$ & $   16$          &               & $48$ & $77309411329$            & $2^{29} \cdot 3^1$ & $2^{29} \cdot 3^2$ & $             2^{30}$    & $30$ \\
$18$ & $18433$                            & $2^{ 8} \cdot 3^2$ & $2^{ 8} \cdot 3^2$ & $     9 \cdot 2^{ 9\pz}$ & $   16$          &               & $49$ & $2^{2}$                  &         ---\pz     & $2^{29} \cdot 3^2$ &         ---\pz\pz        & $30$ \\
$19$ & $101377$                           & $2^{ 9} \cdot 3^2$ & $2^{ 9} \cdot 3^2$ & $    99 \cdot 2^{ 9\pz}$ & $   16$          &               & $50$ & $5469640851457$          & $2^{30} \cdot 3^1$ & $2^{30} \cdot 3^2$ & $   849 \cdot 2^{30}$    & $30$ \\
$20$ & $2424833$                          & $2^{10} \cdot 3^0$ & $2^{10} \cdot 3^2$ & $    37 \cdot 2^{16}$    & $   16$          &               & $51$ & $28114855919617$         & $2^{31} \cdot 3^1$ & $2^{31} \cdot 3^2$ & $  3273 \cdot 2^{30}$    & $30$ \\
$21$ & $12289$                            & $2^{11} \cdot 3^1$ & $2^{11} \cdot 3^2$ & $             2^{ 9\pz}$ & $   16$          &               & $52$ & $1095981164658689$       & $2^{31} \cdot 3^0$ & $2^{31} \cdot 3^2$ & $127589 \cdot 2^{33}$    & $33$ \\
$22$ & $974849$                           & $2^{12} \cdot 3^0$ & $2^{12} \cdot 3^2$ & $   119 \cdot 2^{13}$    & $   16$          &               & $53$ & $2^{3}$                  &         ---\pz     & $2^{31} \cdot 3^2$ &         ---\pz\pz        & $33$ \\
$23$ & $114689$                           & $2^{13} \cdot 3^0$ & $2^{13} \cdot 3^2$ & $     7 \cdot 2^{14}$    & $   16$          &               & $54$ & $87211$                  & $2      \cdot 3^3$ & $2^{31} \cdot 3^3$ & $  2907 \cdot 2^{ 0\pz}$ & $33$ \\
$24$ & $39714817$                         & $2^{14} \cdot 3^1$ & $2^{14} \cdot 3^2$ & $   101 \cdot 2^{12}$    & $   16$          &               & $55$ & $5566277615617$          & $2^{32} \cdot 3^3$ & $2^{32} \cdot 3^3$ & $     3 \cdot 2^{32}$    & $33$ \\
$25$ & $1179649$                          & $2^{15} \cdot 3^2$ & $2^{15} \cdot 3^2$ & $     9 \cdot 2^{16}$    & $   16$          &               & $56$ & $25048249270273$         & $2^{33} \cdot 3^3$ & $2^{33} \cdot 3^3$ & $    81 \cdot 2^{34}$    & $34$ \\
$26$ & $7908360193$                       & $2^{15} \cdot 3^2$ & $2^{15} \cdot 3^2$ & $   419 \cdot 2^{20}$    & $   20$          &               & $57$ & $2$                      &         ---\pz     & $2^{33} \cdot 3^3$ &         ---\pz\pz        & $34$ \\
$27$ & $2^{4}$                            &         ---\pz     & $2^{15} \cdot 3^2$ &         ---\pz\pz        & $   20$          &               & $58$ & $942556342910977$        & $2^{34} \cdot 3^3$ & $2^{34} \cdot 3^3$ & $  1143 \cdot 2^{37}$    & $37$ \\
$28$ & $171048961$                        & $2^{16} \cdot 3^2$ & $2^{16} \cdot 3^2$ & $  1305 \cdot 2^{15}$    & $   20$          &               & $59$ & $2^{3}$                  &         ---\pz     & $2^{34} \cdot 3^3$ &         ---\pz\pz        & $37$ \\
$29$ & $786433$                           & $2^{17} \cdot 3^1$ & $2^{17} \cdot 3^2$ & $             2^{16}$    & $   20$          &               & $60$ & $206158430209$           & $2^{35} \cdot 3^1$ & $2^{35} \cdot 3^3$ & $             2^{33}$    & $37$ \\
$30$ & $14155777$                         & $2^{18} \cdot 3^2$ & $2^{18} \cdot 3^2$ & $    27 \cdot 2^{18}$    & $   20$          &               & $61$ & $2748779069441$          & $2^{37} \cdot 3^0$ & $2^{37} \cdot 3^3$ & $5      \cdot 2^{39}$    & $39$ \\
$31$ & $13631489$                         & $2^{19} \cdot 3^0$ & $2^{19} \cdot 3^2$ & $             2^{20}$    & $   20$          &               & $62$ & $2^{2}$                  &         ---\pz     & $2^{37} \cdot 3^3$ &         ---\pz\pz        & $39$ \\
\bottomrule
\end{tabular}
\end{table}
\end{landscape}
\clearpage
}

To be more precise: For a given~$i$, we write $M_i = 2^{u_i} 3^{v_i} M_i'$ where
$M_i'$ is coprime to~$6$. As we noted in Section~\ref{S:prelim}, there are 
$u_i + O_2(M_i)$ distinct powers of $2$ modulo~$M_i$, and $v_i + O_3(M_i)$
distinct powers of~$3$. For each $M_i$ in turn, we set $M=M_i$ and compute the
solutions $(x,a_1,\ldots,a_n)$ to
\begin{equation}
\label{EQ:sumof2smodM}
\begin{cases}
&3^x \equiv 2^{a_1} + \cdots + 2^{a_n}  \bmod M\\
&0\le x < v + O_3(M)\\
&0 = a_1\le \cdots \le a_n < u + O_2(M),
\end{cases}
\end{equation}
with the added condition that for every pair $(j,k)$ of indices with $j\neq k$,
if $a_j$ and $a_k$ are both less than $u_i$, then $a_j\neq a_k$. This last
condition reflects the fact that if $a < u_i$, then $2^a$ is a determinate power
of $2$ in $\ZZ/M_1\ZZ$, and the right-hand side exponents in the solutions to 
equation~\eqref{EQ:sumof2s} are required to be distinct. (Note that the upper
bounds given in~\eqref{EQ:sumof2smodM} have the effect of keeping us in line 
with Convention~\ref{convention}.)

For $M_1 = 2^4 \cdot 7\cdot 73$ we compute the solutions 
to~\eqref{EQ:sumof2smodM} by brute force. The powers of $2$ in $\ZZ/M_1\ZZ$ are 
$2^0$ through~$2^{12}$. To every $n$-tuple $(a_1,\ldots,a_n)$ of exponents 
between $0$ and $12$ with $0 = a_1\le \cdots \le a_n$, we can associate the 
$13$-tuple $(b_0,\ldots,b_{12})$, where $b_i$ is the number of $a_j$ that are 
equal to~$i$. Then instead of enumerating all of the $n$-tuples 
$(a_1,\ldots,a_n)$, we can simply run through all of the $13$-tuples 
$(b_0,\ldots,b_{12})$ of non-negative integers such that
\[
b_0 + \cdots + b_{12} = n
\]
and 
\[
b_0 = 1, \quad b_1\le 1, \quad b_2\le 1, \quad\text{and\ } b_3\le 1.
\]
When we find such a $13$-tuple with the additional property that $\sum b_j 2^j$ 
is congruent to $3^x$ modulo $M_1$ for one of the $12$ powers of $3$ modulo~$M_1$, 
we can compute the associated $n$-tuple $(a_1,\ldots,a_n)$ and add
$(x,a_1,\ldots,a_n)$ to our list of solutions of equation~\eqref{EQ:sumof2smodM}
with $M = M_1$. We obtain all solutions to the equation in this way.

Now suppose we have a list of solutions to~\eqref{EQ:sumof2smodM} with 
$M = M_{i-1}$, and we want to create the list of solutions with $M = M_i$, where
$M_i = m_i M_{i-1}$. Write $M_i = 2^{u_i} 3^{v_i} M_i'$ with $M_i'$ coprime
to~$6$. For each solution $(x, a_1, \ldots, a_n)$ to the problem modulo~$M_{i-1}$,
we go through the following steps.

\subsection*{Step One:}
\emph{Compute the powers of $2$ in $\ZZ/M_i\ZZ$ that lift the 
$2^{a_j} \in \ZZ/M_{i-1}\ZZ$.} 
      
For each $j=1, \ldots, n$, we compute a list $A_{j}$ of the values of $a'$ with
$0 \le a' < u_i + O_2(M_i)$ such that $2^{a'} \equiv 2^{a_j} \bmod M_{i-1}$.

\subsection*{Step Two:}
\emph{Compute the number of powers of $3$ in $\ZZ/M_i\ZZ$ that lift 
$3^x \in \ZZ/M_{i-1}\ZZ$.}
      
Let $\chi$ denote the number of values of $x'$ with $0 \le x' < v_i + O_3(M_i)$
such that $3^{x'} \equiv 3^{x}\bmod M_{i-1}$. If $3^x$ is a determinate power of
$3$ modulo~$M_{i-1}$, then $\chi = 1$. If $3^x$ is an indeterminate power of $3$
modulo~$M_{i}$, then $\chi = O_3(M_i)/O_3(M_{i-1})$. And if $3^x$ is 
indeterminate modulo $M_{i-1}$ but determinate modulo $M_{i}$, then
$\chi = 1 + O_3(M_i)/O_3(M_{i-1})$.

\subsection*{Step Three:}
\emph{Compute the lifted solutions.}

We compute lifted solutions in one of two ways; to decide between the two
methods, we check to see whether $\chi >  \prod_{j=1}^n \#A_j$ and whether $m_i$
is a prime that does not divide $6M_{i-1}$. If both these conditions hold, we 
say we are in the \emph{unbalanced} case, and if not we say we are in the 
\emph{balanced} case.

\begin{enumerate}
\item \emph{The unbalanced case.} In this case we must have $\chi>1$, so $3^x$ 
      is an indeterminate power of $3$ modulo~$M_{i-1}$; also, in this case we
      have $v_i = v_{i-1}$ because $m_i\neq 3$. We proceed as follows, for each 
      $n$-tuple $(a_1', \ldots, a_n')$ in $A_1\times\cdots\times A_n$:
      \begin{enumerate}
      \item \emph{Compute the right-hand side sum.}
            Set $s \colonequals \sum_{j} 2^{a_j'}$.
      \item \emph{Check to see whether the right-hand side sum is a power 
            of~$3$ modulo~$M_i$.} To check to see whether there is a power of
            $3$, say $3^{x'}$, with $3^{x'}\equiv s\bmod M_i$, we use discrete
            logarithms as follows.

            \qquad Let $g$ be a generator of the group of units of $\ZZ/m_i\ZZ$,
            let $z$ be the smallest non-negative integer with
            $g^z \equiv s\bmod m_i$, and let $y$ be the smallest positive
            integer with $g^y \equiv 3\bmod m_i$, so that $z$ and $y$ are
            discrete logarithms of $s$ and of $3$ with respect to the base~$g$.
            If there is an $x'$ such that $3^{x'}\equiv s\bmod M_i$, then for
            this $x'$ we have $3^{x'}\equiv s\bmod m_i$, so we must have 
            ${x' y \equiv z \bmod (p-1)}$; for this $x'$ we have
            $3^{x'}\equiv s\bmod 2^{v_{i-1}} M_{i-1}$, so we must have 
            $x' \equiv x\bmod O_3(M_{i-1})$; and for this $x'$ we have
            $3^{x'} \equiv 3^x \equiv 0\bmod 3^{v_i}$, so we must have 
            $x' \ge v_{i}$. Conversely, any $x'$ that satisfies these three
            conditions will also satisfy $3^{x'}\equiv s \bmod M_i$.
            
            \qquad For primes $m_i$ of the size we are considering, the 
            computation of the discrete logarithms $z$ and $y$ is easily done by
            the computer algebra package Magma, in which we have written our 
            code. It is also a straightforward matter to compute the values of
            $x'$ that meet the three conditions, if any exist. 
            
            \qquad For each $x'$ that we find, we add $(x', a'_1, \ldots, a'_n)$
            to our list of solutions of equation~\eqref{EQ:sumof2smodM} with 
            $M = M_i$.
      \end{enumerate}
      The time required to carry out this step is proportional to the number of 
      $n$-tuples  $(a'_1, \ldots, a'_n)$ that we have to consider, which 
      is~$\prod \#A_i$.
\smallskip      
\item \emph{The balanced case.} We proceed as follows.
      \begin{enumerate}
      \item \emph{Compute the left-hand side lifts.} We compute the set $X$ of 
            the values of $x'$ with $0 \le x' < v_i + O_3(M_i)$ such that 
            $3^{x'} \equiv 3^{x} \bmod M_{i-1}$.  
      \item \emph{Group the variables into two balanced sets.} Compute the
            value of $k$ so that the product $\#X \cdot \prod_{j\le k} \#A_j$ 
            and the product $\prod_{j>k} \#A_j$ are as close in size as 
            possible. 
      \item \emph{Compute the lifts of the variables in each grouping.} We make
            two lists. The first is the list of all $(k+2)$-tuples
            \[ 
            (3^{x'} - 2^{a_1'} - \cdots - 2^{a_k'}, x', a_1', \ldots, a_k') 
            \]
            for all 
            $(x', a_1', \ldots, a_k') \in X \times A_1 \times \cdots \times A_k$,
            where we view the first entry of the tuple as an element of 
            $\ZZ/M_i\ZZ$. The second is the list of all $(n-k+1)$-tuples
            \[ 
            (2^{a_{k+1}'} + \cdots + 2^{a_n'}, a_{k+1}', \ldots, a_n')
            \]
            for all 
            $(a_{k+1}', \ldots, a_n') \in A_{k+1} \times \cdots \times A_n$,
            where again we view the first entry as an element of $\ZZ/M_i\ZZ$.
      \item \emph{Compare the lists for matching values.} Sort each of these 
            lists according to the value of the first entry of each tuple, and
            then compare the two sorted lists to find all pairs of elements, one
            from the first list and one from the second, whose first entries are
            equal. Every such pair gives us a solution to
            \[
            3^{x'} \equiv 2^{a_1'} + \cdots + 2^{a_n'} \text{ in } \ZZ/M_i\ZZ
            \]
            that reduces to our original solution in $\ZZ/M_{i-1}\ZZ$. Add each
            such solution to our list of solutions of 
            equation~\eqref{EQ:sumof2smodM} with $M = M_i$.
      \end{enumerate}
      The time it takes to carry out this step is proportional to the larger of 
      ${\#X \cdot \prod_{j\le k} \#A_j}$ and $\prod_{j>k} \#A_j$. If these two 
      numbers are somewhat balanced, the time required for this step will be 
      roughly proportional to the square root of 
      ${\#X \cdot \prod_{j\le n} \#A_j}$.
\end{enumerate}

Once we have computed all of the solutions to equation~\eqref{EQ:sumof2smodM}  
with $M = M_i$ by this method, we check to see whether all of the powers of $2$
that occur anywhere on our list are determinate. If they are not, then we 
increase $i$ by $1$ and iterate the procedure. If they are, then for each 
solution to~\eqref{EQ:sumof2smodM} with $M = M_i$, we can check to see whether 
the (unique) lifts of the terms in the right-hand side of~\eqref{EQ:sumof2smodM}
to powers of $2$ in $\ZZ$ add up to a power of~$3$. In this way, we hope to find
all solutions to~\eqref{EQ:sumof2s}.

\begin{proof}[Proof of Theorem~\textup{\ref{T:POWERSOF2}}]
We ran through the procedure described above for all values of $n$ from $3$
to~$22$. For each~$n$, the procedure did terminate before we ran out of values 
of $M_i$, so we successfully found all solutions to equation~\eqref{EQ:sumof2s}
for $n\le 22$. We found that the binary representation of $3^x$ has at most
twenty-two bits equal to $1$ exactly when $x\le 25$.
\end{proof}

In Table~\ref{T:timing2}, we give for each $n$ the value of $i$ for which the 
modulus $M_i$ gave us all solutions to the equation. We also give the total time
for the computation. As mentioned earlier, the programs we used to implement 
this computation were written in Magma and are available as
supplementary material attached to the ArXiv version of this 
paper~\cite{DimitrovHowe2021}, as well as on the second author's web site.

\begin{table}
\label{T:timing2}
\centering
\caption{For each $n$, we list the value of $i$ such that our procedure for 
         solving equation~\eqref{EQ:sumof2s} iterated up to the modulus $M_i$ 
         from Table~\ref{T:2}. We also give the wall-clock time it took for the
         computation to complete on a 2.8 GHz Quad-Core Intel Core i7 with 
         16GB RAM running Magma V2.23-1 on Mac OS 11.2.3. For $n\ge 20$ the 
         computation was split into parts that were run by separate processes;
         the time given is the sum of the wall-clock times for each process.}
\begin{tabular}{r r c c r r c}
\toprule
 $n$ &  $i$ & Time (sec) &\hbox to 1em{} &  $n$ &  $i$ & Time (sec)         \\
\cmidrule{1-3}\cmidrule{5-7} 
 $3$ & $10$ &  $0.01$    &               & $13$ & $37$ & $\phantom{0000}19$ \\
 $4$ & $10$ &  $0.02$    &               & $14$ & $45$ & $\phantom{0000}52$ \\
 $5$ & $14$ &  $0.04$    &               & $15$ & $45$ & $\phantom{000}145$ \\
 $6$ & $14$ &  $0.07$    &               & $16$ & $59$ & $\phantom{000}457$ \\
 $7$ & $14$ &  $0.14$    &               & $17$ & $59$ & $\phantom{00}1469$ \\
 $8$ & $14$ &  $0.29$    &               & $18$ & $62$ & $\phantom{00}5746$ \\
 $9$ & $14$ &  $0.62$    &               & $19$ & $62$ & $\phantom{0}17744$ \\
$10$ & $27$ &  $1.54$    &               & $20$ & $62$ & $\phantom{0}53617$ \\
$11$ & $37$ &  $3.81$    &               & $21$ & $62$ & $          139347$ \\
$12$ & $37$ &  $8.03$    &               & $22$ & $62$ & $          743737$ \\
\bottomrule
\end{tabular}
\end{table}

The procedure we described in the proof of Theorem~\ref{T:POWERSOF2} suggests
the properties we looked for when choosing the factors $m_i$ out of which our 
moduli $M_i$ are built. In the balanced case, we want the sets $A_j$ to be as
small as possible, since the work in the balanced case is roughly on the order
of the square root of the product $\#X \cdot \prod_{j\le n} \#A_j$. Of course,
we'd like $\#X$ to be small as well, but since there are $n$ sets $A_j$ we
concentrate first on them. 

For a given solution $(x,a_1,\ldots,a_n)$ to~\eqref{EQ:sumof2smodM} with 
$M = M_{i-1}$, how large are the~$A_j$? The answer is analogous to the 
computation of the value of $\chi$ given in Step~Two of our procedure. Suppose
we are in the case where $m_i$ is odd. If $2^{a_j}$ is a determinate power of
$2$ modulo~$M_{i-1}$, then $\#A_j = 1$. If $2^{a_j}$ is indeterminate 
modulo~$M_{i-1}$, then it is indeterminate modulo~$M_{i}$ as well because $m_i$
is odd, and we have $\#A_j = O_2(M_i)/O_2(M_{i-1})$. If $m_i$ is coprime to 
$M_{i-1}$, which is the case for all of the values we chose, then $O_2(M_i)$ is
the least common multiple of $O_2(m_i)$ and $O_2(M_{i-1})$. 

The ideal case would be for $O_2(m_i)$ to be a divisor of $O_2(M_{i-1})$, so
that the ratio $O_2(M_i)/O_2(M_{i-1})$ would be~$1$. The next-best case would be 
for $O_2(m_i)$ to divide $2 O_2(M_{i-1})$ but not $O_2(M_{i-1})$, so that 
$O_2(M_i)/O_2(M_{i-1})$ would be~$2$. We were able to stay in these two cases 
for every $i$ with $m_i$ odd, except for $i = 54$, where we have
$O_2(M_i)/O_2(M_{i-1}) = 3$.

For those $i$ for which $O_2(M_i)/O_2(M_{i-1}) = 1$, we can focus more on the
unbalanced case. These $i$ give us the opportunity to build up the number of
powers of $2$ in $O_3'(M_i)$. For example, for $i=13$ we have
$O_2(M_i)/O_2(M_{i-1}) = 1$, and with the value of $m_i$ that we chose, we
increase the $2$-part of the order of $3$ from $2^8$ in $O_3'(M_{i-1})$ to 
$2^{16}$ in $O_3'(M_i)$.

We found our $m_i$ mostly by looking for primes $p$ congruent to $1$ modulo
$2^a 3^b$ for various values of $a$ and~$b$, and computing the orders of $2$ 
and $3$ in $(\ZZ/p\ZZ)^*$.

We make one final note about our choice of the $m_i$. We would also like the 
number of solutions we have to consider at any given stage to be small. This 
becomes especially critical for the larger values of $n$ that we consider. Our 
choices for $m_i$, especially for small~$i$, reflect this. For example, we have 
chosen $m_4$ to be $241\cdot 433$, which puts us in the balanced case with 
$\#A_j = 2$ for most $j$ and with $\#X = 10$. After this $m_4$, we have
$m_5 = 17$, $m_6 = 2^2$, and $m_7 = 38737$. For smaller values of~$n$, it turns
out that it would be faster to take $m_4 = 433$ (which gives us $\#X = 1$), 
$m_5 = 17$, $m_6 = 2^2$, and then to add in a factor of $241$ before moving on 
to $m_7 = 38737$. According to the heuristic mentioned in Step~Three, the time 
it takes to process a solution in the balanced case is very roughly proportional
to $(\#X \cdot \prod_{j\le n} \#A_j)^{1/2}$, so having $\#X$ equal to $1$
instead of $10$ should speed up this step by a factor of  about $\sqrt{10}$. But
for large~$n$, this improved speed for $i=4$ would be outweighed by the extra
time it would take to process the large number of solutions that would make it 
through to the next step. To simplify our exposition, we have simply given one 
single sequence of $m_i$ to use for all~$n$, optimized for large values of~$n$,
even though different choices would have made the program run faster for 
smaller~$n$.

\section{Proof of Theorem~\ref{T:POWERSOF3}}
\label{S:3}

The proof of  Theorem~\ref{T:POWERSOF3} is also computational, and is 
essentially the same as that of Theorem~\ref{T:POWERSOF2}. The sequence of 
moduli we use is given in~Table~\ref{T:3}, and the time it took to run our 
program for $n$ up to $24$ is given in Table~\ref{T:timing3}. The only other 
comment we make here is that if $n$ is odd and greater than~$1$, then there are
no solutions to equation~\eqref{EQ:sumof3s}, because no power of $2$ (other 
than~$1$) can be written as the sum of an odd number of powers of~$3$.
\qed

\begin{table}
\label{T:3}        
\centering
\caption{Data for the factors $m_i$ and the moduli $M_i = \prod_{j\le i} m_j$ 
         used in the proof of Theorem~\ref{T:POWERSOF3}. The notation in the
         table headings is as in Notation~\ref{N:23}.}
\begin{tabular}{r r r r r c}
\toprule
 $i$ & $m_i$                            & $O_3(m_i)$          & $O_3(M_i)$      & $O_2'(m_i)$            & $v_3(O_2'(M_i))$ \\
\midrule
 $1$ & $2 \cdot 3^4 \cdot 13 \cdot 757$ & $          3^2$     & $          3^2$ & $28    \cdot 3^{3\pz}$ & $\pz3$ \\
 $2$ & $7 \cdot 19 \cdot 37$            & $2   \cdot 3^2$     & $2   \cdot 3^2$ & $4     \cdot 3^{2\pz}$ & $\pz3$ \\
 $3$ & $5 \cdot 73$                     & $2^2 \cdot 3^{\pz}$ & $2^2 \cdot 3^2$ & $4     \cdot 3^{2\pz}$ & $\pz3$ \\
 $4$ & $530713$                         & $2^2 \cdot 3^2$     & $2^2 \cdot 3^2$ & $91    \cdot 3^{6\pz}$ & $\pz6$ \\
 $5$ & $3^3$                            &       ---\pz        & $2^2 \cdot 3^2$ &         ---\pz\pz      & $\pz6$ \\
 $6$ & $41 \cdot 6481$                  & $2^3 \cdot 3^{\pz}$ & $2^3 \cdot 3^2$ & $20    \cdot 3^{4\pz}$ & $\pz6$ \\
 $7$ & $282429005041$                   & $2^3 \cdot 3^2$     & $2^3 \cdot 3^2$ & $66430 \cdot 3^{12}  $ & $  12$ \\
 $8$ & $3^6$                            &       ---\pz        & $2^3 \cdot 3^2$ &         ---\pz\pz      & $  12$ \\
\bottomrule
\end{tabular}
\end{table}

\begin{table}
\label{T:timing3}
\centering
\caption{For each $n$, we list the value of $i$ such that our procedure for 
         solving equation~\eqref{EQ:sumof3s} iterated up to the modulus $M_i$
         from Table~\ref{T:3}. We also give the wall-clock time it took for the
         computation to complete on a 2.8 GHz Quad-Core Intel Core i7 with 
         16GB RAM running Magma V2.23-1 on Mac OS 11.2.3.}
\begin{tabular}{r r c c r r c}
\toprule
$n$  & $i$ & Time (sec) &\hbox to 1em{} &  $n$ & $i$ & Time (sec)\\
\cmidrule{1-3}\cmidrule{5-7} 
 $4$ & $5$ &     $0.01$ &               & $16$ & $8$ & $\phantom{0000}14$ \\
 $6$ & $5$ &     $0.01$ &               & $18$ & $8$ & $\phantom{0000}84$ \\
 $8$ & $5$ &     $0.07$ &               & $20$ & $8$ & $\phantom{000}789$ \\
$10$ & $8$ &     $0.23$ &               & $22$ & $8$ & $\phantom{00}9792$ \\
$12$ & $8$ &     $0.92$ &               & $24$ & $8$ & $          140036$ \\
$14$ & $8$ &     $3.44$ \\
\bottomrule
\end{tabular}
\end{table}

\bibliographystyle{hplaindoi}
\bibliography{elementaryEDE}

\end{document}